\documentclass{birkmult}
\usepackage{amsmath,amssymb,verbatim,amsthm}
\usepackage[mathscr]{eucal}
\usepackage[draft]{epsfig}

\def\R{\mathbb R}
\def\C{\mathbb C}

\def\restr{\negthickspace \mid}
\def\di{\partial}
\def\&{\wedge}

\def\ri{\mathrm i}
\def\ihat{\hat\imath}
\def\jhat{\hat\jmath}
\def\khat{\hat{k}}
\def\Ad{\operatorname{Ad}}

\def\w{\omega}

\def\su{\mathfrak{su}}
\def\h{\mathfrak{h}}
\def\k{\mathfrak{k}}

\def\Lt{\widetilde{\Lambda}}

 \newtheorem{thm}{Theorem}[section]
 \newtheorem{cor}[thm]{Corollary}
 \newtheorem{lem}[thm]{Lemma}
 \newtheorem{prop}[thm]{Proposition}
 \theoremstyle{definition}
 \newtheorem{defn}[thm]{Definition}
 \theoremstyle{remark}
 \newtheorem{rem}[thm]{Remark}
 
 \numberwithin{equation}{section}

\def\FL{F^\lambda}

\begin{document}
\title[Symmetric Pseudospherical Surfaces I]{Symmetric Pseudospherical Surfaces I:\\ General Theory}
\author[Dorfmeister]{Josef F. Dorfmeister}
\address{Zentrum Mathematik\\
Technische Universit\"{a}t M\"{u}nchen\\
D-85747 Garching bei M\"{u}nchen\\ Germany} 
\email{dorfm@ma.tum.de} 
\author[Ivey]{Thomas Ivey}
\address{Department of Mathematics\\
College of Charleston\\
Charleston, SC  29424\\ USA}
\email{iveyt@cofc.edu} 
\author[Sterling]{Ivan Sterling}
\address{Mathematics and Computer Science Department\\ St Mary's College of Maryland\\ St Mary's City, MD 20686-3001\\ USA}
\email{isterling@smcm.edu} 
\date{February 28, 2009}
\dedicatory{In Memory of Katsumi Nomizu}

\begin{abstract}
We apply the loop group
method developed by Zakharov-Shabat \cite{ZS}, Terng-Uhlenbeck \cite{TU} and Toda \cite{T} to the study of symmetries of pseudospherical surfaces (ps-surfaces) in $\mathbb{R}^3$.  
In this paper (part I) we consider the general theory, while in a second paper (part II) we will study special cases. 
\end{abstract}
\maketitle

\section{Introduction}
In this paper, we study symmetries of pseudospherical surfaces (ps-surfaces, i.e. surfaces with Gauss curvature $K=-1$) in $\mathbb{R}^3$ via the loop group method developed by Zakharov-Shabat \cite{ZS}, Terng-Uhlenbeck \cite{TU} and Toda \cite{T}.   
One of the motivations for studying symmetries is to develop a theory for non-finite-type ps-surfaces.  
(A rather complete investigation of ps-surfaces of finite-type is Melko-Sterling \cite{MS}.)  
In particular, using methods in Part I we will exhibit examples with discrete rotational symmetry about an axis.  These examples contain points which have properties similar to umbilic points.  We believe these examples will help to develop a theory of ps-surfaces of non-finite type.  

In \S\ref{sec:lgps} of the paper we review the main results of Toda's algorithm as it has been used computationally for several years.  First, we
discuss the 1:1 correspondence (up to rigid motions) between ps-surfaces parametrized by asymptotic lines and pairs of \textit{normalized potentials}. 
In preparation for understanding the relationships between symmetries at various levels, we review the precise correspondence between four levels of
description for a ps-surface: the immersion $f:D \to \mathbb{R}^3$
itself, the extended orthonormal frame $F$, the extended $SU(2)$-valued
frame $U$, and the normalized potential pair.  The main results in this section are the construction of normalized potentials in \eqref{duplusx} and
\eqref{duminusy}, and the converse in Theorem \ref{converse}.
We also address the 
questions of uniqueness and differentiability, and introduce generalized potentials.

In \S\ref{sec:startsym} we study symmetries of ps-surfaces, frames and potentials.  Our study is similar to that of Dorfmeister and Haak's study of symmetries of constant mean curvature surfaces \cite{DH1},\cite{DH2}.  Our basic assumption is that there is a rigid motion $R:\R^3 \to \R^3$ and a diffeomorphism
$\gamma:D \to D$ such that
\begin{equation*}
f \circ \gamma = R \circ f.
\end{equation*}
In particular we address the issue of how group actions on the surfaces relate to group actions on the space of general potentials.  The main results in this section are Propositions \ref{SeppsPropA} and \ref{SeppsPropB}.

In the second paper \cite{DIS2} we will study special cases, both old and new, including symmetries via the fundamental group or rotational invariance.  New examples include several with discrete rotational symmetry.  One such example is shown in Figures 1 and 2 of this paper.
\section{Loop Groups and Pseudospherical Surfaces} \label{sec:lgps}

Here, we want to summarize the loop group method for constructing ps-surfaces.  We aim to give a compact exposition, so some proofs will be omitted. 
 
\subsection{Ps-Surface to Darboux Frame}\label{ps2Darboux}
 We begin by reviewing some well-known facts about ps-surfaces, beginning with the fact that the asymptotic lines form a Chebyshev net, and the angle between them satisfies the sine-Gordon equation:

\begin{thm} Let $f:D \to \R^3$ be an oriented immersed ps-surface.  Near any point of $D$, there are coordinates $x$ and $y$ such that $\di f/\di x$ and $\di f/\di y$ are unit vectors and asymptotic directions, and  $\di f/\di x\times \di f/\di y$ agrees with the orientation. Then the counterclockwise angle $\phi$ from $\di f/\di x$ to $\di f/\di y$ satisfies the sine-Gordon equation
\begin{equation}\label{SGE}
\phi_{xy} = \sin \phi.
\end{equation}
Let $\theta=\phi/2$.
If we define the Darboux frame\footnote{That is, a moving frame along the surface where the first two  vectors are principal directions.} (see figure below)
$$e_1 = \tfrac12 \sec{(\theta)} (\di f/\di x + \di f/\di y), \quad
e_2 = \tfrac12 \csc{(\theta)}(\di f/\di y-\di f/\di x), \quad
e_3 = e_1 \times e_2,
$$
then the orthogonal matrix $\widetilde{F}$ whose columns are $e_1, e_2, e_3$ satisfies
\begin{equation}\label{so3sys}
\dfrac{\di \widetilde{F}}{\di x} =  \widetilde{F} \begin{bmatrix}
0 & \theta_x & -\sin \theta \\
-\theta_x & 0 & -\cos \theta \\
\sin \theta & \cos \theta & 0
\end{bmatrix},\qquad
\dfrac{\di \widetilde{F}}{\di y} = \widetilde{F} \begin{bmatrix}
0 & -\theta_y & -\sin \theta \\
\theta_y & 0 & \cos \theta \\
\sin \theta & -\cos \theta & 0
\end{bmatrix}.
\end{equation}
\end{thm}
For this calculation, see \S6.4 in \cite{IL}.

\begin{center}
\includegraphics[height=1.5in]{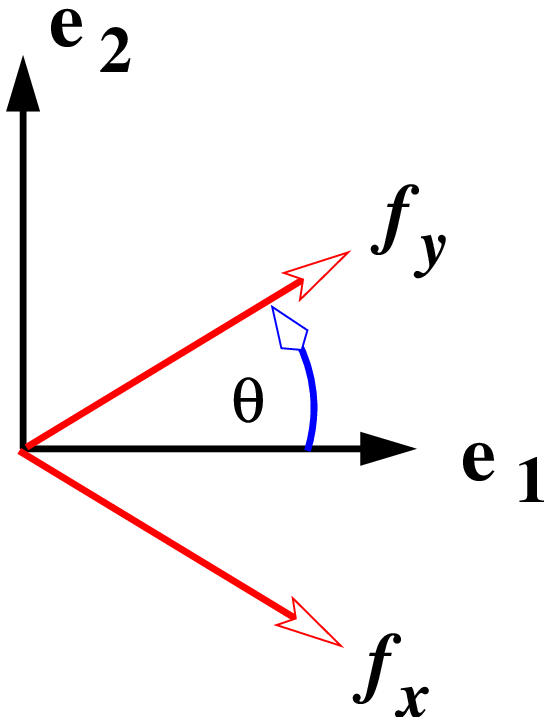}
\end{center}

We also have the converse:

\begin{thm}  Let $D$ be a simply connected open set in $\R^2$ and let $\widetilde{F}(x,y)$ be an $SO(3)$-valued function on $D$ that satisfies \eqref{so3sys} for some smooth function $\theta(x,y)$ on $D$.  Then $\phi=2\theta$ satisfies the sine-Gordon equation, and there is a map $f:D\to \R^3$ which is an immersion at points where $\sin\phi \ne0$, whose image is a ps-surface with Darboux frame given by the columns of $\widetilde{F}$.
\end{thm}

\subsection{Darboux Frame to Extended Frame}\label{Darboux2extended}
Rather than using a Darboux frame, it will be more convenient to use a frame that includes the unit vector $\di f/\di x$.  Accordingly, we let $F$ denote the frame obtained by rotating the first two vectors of the Darboux frame $\widetilde{F}$ through the clockwise angle $\theta$, so that
$$F = \widetilde{F} \begin{bmatrix} \cos\theta & \sin \theta & 0 \\ -\sin\theta & \cos\theta & 0 \\ 0 & 0 & 1\end{bmatrix}.$$
 Then
\begin{equation}\label{so3syst}
\dfrac{\di F}{\di x} =  F
\begin{bmatrix}
0 & \phi_x & 0 \\
-\phi_x & 0 & -1 \\
0 & 1 & 0
\end{bmatrix},\qquad
\dfrac{\di F}{\di y} = F
\begin{bmatrix}
0 & 0 & -\sin \phi \\
0 & 0 & \cos \phi \\
\sin \phi & -\cos \phi & 0
\end{bmatrix}.
\end{equation}

The sine-Gordon equation for $\phi$ is derived as the compatibility condition for the overdetermined system \eqref{so3syst}, by setting $\di (F_x)/\di y = \di (F_y)/\di x$.

The sine-Gordon equation is invariant under the 1-parameter group of Lie symmetry transformations of the form $T^\lambda(x,y) = (\lambda x, \lambda^{-1} y)$, $\lambda > 0$.  Hence,  $F^\lambda = F \circ T^\lambda$ will satisfy an overdetermined system with the same compatibility condition; in fact,
\begin{equation}\label{ftlsys}
\dfrac{\di F^\lambda}{\di x} = F^\lambda
\begin{bmatrix}
0 & \phi_x & 0 \\
-\phi_x & 0 & -\lambda \\
0 & \lambda & 0
\end{bmatrix} ,\ 
\dfrac{\di F^\lambda}{\di y} =\dfrac1\lambda F^\lambda
\begin{bmatrix}
0 & 0 & -\sin \phi \\
0 & 0 &  \cos \phi \\
\sin \phi & -\cos \phi & 0
\end{bmatrix} .
\end{equation}
For any fixed $\lambda \in \R^+$, $F^\lambda$ is an orthonormal frame for a ps-surface; these ps-surfaces make up an {\em associated family} of ps-surfaces, which includes the original ps-surface when $\lambda=1$.

\subsection{Lifting the Extended Frame}\label{liftextended}
It will also be convenient to work with matrices in the Lie group $SU(2)$ instead of $SO(3)$.
Recall that we can identify $\R^3$ with the Lie algebra $\su(2)$ 
in a way that the adjoint action of $SU(2)$
corresponds to rotations in $\R^3$, with every rotation in $SO(3)$ being realized by $\Ad(g)$
for two possible elements $g \in SU(2)$, differing by a minus sign.  This gives a double cover
$\delta: SU(2) \to SO(3)$.  Provided that $D$ is simply-connected, we can choose a well-defined
lift of $F^\lambda$ into $SU(2)$, and we let $U:D \to SU(2)$ denote this mapping, which we will also refer to as the extended frame of the surface.

To specify the lifting, we fix an identification of the standard basis vectors $\ihat,\jhat,\khat$ for $\R^3$ with matrices in $\su(2)$ given by
$$\ihat \leftrightarrow \frac12 \begin{bmatrix} 0 & \ri \\ \ri & 0\end{bmatrix}, \quad
\jhat \leftrightarrow \frac12 \begin{bmatrix} 0 & -1 \\1 & 0 \end{bmatrix}, \quad
\khat \leftrightarrow \frac12\begin{bmatrix} \ri  & 0 \\ 0 & -\ri\end{bmatrix}.
$$
(This identification has the virtue that the cross-product in $\R^3$ corresponds exactly to the Lie bracket in $\su(2)$.)  Let $e_1,e_2,e_3$ denote the columns of $F^\lambda = \delta(U)$. Then
\begin{equation}\label{econj}
e_1 = U \ihat U^{-1}, \quad e_2 = U \jhat U^{-1}, \quad e_3 = U \khat U^{-1},
\end{equation}
where we are now tacitly identifying the vectors $e_i$ and $\ihat,\jhat,\khat$ with their matrix counterparts.  We can use the differential equations satisfied by $F^\lambda$ to deduce the
components of
$$\omega = U^{-1} dU.$$
The system \eqref{ftlsys} implies that $\di e_1/\di x = -\phi_x e_2$ and
$\di e_3/\di x = -\lambda e_2$.
Differentiating \eqref{econj} shows that $de_1 = U [ \omega ,\ihat] U^{-1}$
and $de_3 = U[\omega,\khat]U^{-1}$.
Thus, the $dx$ coefficient in $\omega$ must
be $-\phi_x \khat + \lambda \ihat$.  Similarly, \eqref{ftlsys}
implies that $\di e_1/\di y = \lambda^{-1}\sin \phi\, e_3$ and
$\di e_2/\di y = -\lambda^{-1}\cos\phi\, e_3$, so the $dy$
coefficient in $\omega$ must be
$$\lambda^{-1}(-\sin\phi\, \jhat - \cos\phi\, \ihat)
=-\dfrac{1}{2\lambda} \begin{bmatrix}
0 & \ri \cos\phi - \sin\phi \\
\ri\cos \phi +\sin \phi & 0 \end{bmatrix}
= -\dfrac{\ri}{2\lambda}\begin{bmatrix}
0 & e^{\ri\phi} \\  e^{-\ri \phi} & 0 \end{bmatrix}.
$$
Thus, $U$ satisfies
\begin{equation}\label{su2sys}
\dfrac{\di U}{\di x} = \frac{\ri}2 U
\begin{bmatrix} -\phi_x & \lambda \\ \lambda & \phi_x \end{bmatrix} , \qquad
\dfrac{\di U}{\di y} = -\frac{\ri}{2\lambda} U
\begin{bmatrix}
0 & e^{\ri\phi} \\  e^{-\ri \phi} & 0 \end{bmatrix}.
\end{equation}

The compatibility condition
for this system is equivalent to requiring that $\omega$, as an $\su(2)$-valued 1-form on $D$,
satisfy the Maurer-Cartan equation
$$d\omega = -\omega \wedge \omega$$
for each $\lambda$.
Notice that the $\lambda$-dependent parts of $\omega$ are on the off-diagonal only.  This means
that $\omega$ is a 1-form on $D$ taking values in the loop algebra
\begin{equation}\label{algloopdef}
\Lambda\su(2) = \{ X: \R^* \to \su(2) |\, X(-\lambda) = \Ad(\sigma_3)\cdot X(\lambda) \},
\quad\text{where } \sigma_3 = \left( \begin{smallmatrix} 1 & 0 \\ 0 & -1\end{smallmatrix}\right).
\end{equation}
Likewise, $U$ is a map from $D$ to the loop group
$$\Lambda SU(2) = \{ g: \R^* \to SU(2) | \,g(-\lambda) = \Ad(\sigma_3)\cdot g(\lambda) \}.$$
(loops satisfying the $\Ad(\sigma_3)$ condition are sometimes referred to as {\em twisted}.)
We will be specifically interested in those subgroups, denoted by $\Lt\su(2)$ and $\Lt SU(2)$
respectively, consisting of loops which extend to $\C^*$ as  analytic functions of $\lambda$.
(Note, however, that such extensions will take values in $\mathfrak{sl}(2,\C)$ and $SL(2,\C)$ respectively.)
In fact, the goal of the method is to recover such loops from analytic data specified along
a pair of characteristic curves in $D$.

Within the group of loops that extend analytically to $\C^*$, we define subgroups of loops which extend
to $\lambda=0$ or $\lambda=\infty$:
\begin{align*}
\Lambda^+ SU(2) &= \{ g \in \Lt SU(2) | g = g_0 + \lambda g_1 + \lambda^2 g_2 + \ldots \} \\
\Lambda^-SU(2) &=\{ g \in \Lt SU(2) | g = g_0 + \lambda^{-1} g_1 + \lambda^{-2} g_2 + \ldots \}
\end{align*}
Within these, we let $\Lambda^+_* SU(2)$ and $\Lambda^-_* SU(2)$
be the subgroups of loops where $g_0$ is the identity matrix.

\subsection{Extended Frame to Normalized Potential Pair}
\label{extended2normalized}
A key tool we will use is
\begin{thm}[Birkhoff Decomposition \cite{G}, \cite{T}]
The multiplication maps
$$\Lambda^+_* SU(2) \times \Lambda^- SU(2) \to \Lt SU(2), \qquad
\Lambda^-_* SU(2) \times \Lambda^+ SU(2) \to \Lt SU(2)$$
are diffeomorphisms.
\end{thm}

\begin{rem} In general, the Birkhoff decomposition theorem asserts that
the multiplication maps are diffeomorphisms onto an open dense subset, known as the {\bf big cell}.  However, it follows from the recent result of Brander \cite{B} that in the case of compact semisimple Lie groups like $SU(2)$, the big cell is everything.
\end{rem}

We apply both Birkhoff decompositions to $U$, giving
\begin{equation}\label{usplit}
U = U^X_+ V_- = U^Y_- V_+,
\end{equation}
where $U^X_+,U^Y_- \in \Lambda_*^\pm SU(2)$, $V_-(x,y) \in \Lambda^- SU(2)$, and $V_+(x,y) \in \Lambda^+ SU(2)$.
The superscripts in $U^X_+$ and $V^Y_-$ are justified by the following important insight:
\begin{prop}\label{onlyXonlyY} $U^X_+$ does not depend on $y$, and $U^Y_-$ does not depend on  $x$.
\end{prop}
\begin{proof}  From \eqref{su2sys}, it follows that $\omega = U^{-1} dU$ has
the form
$$\omega = A\, dx + B\, dy, \qquad A = A_0 + \lambda A_1, \quad B = B_1 \lambda^{-1},$$
where $A_0,A_1, B_1$ depend only on $x$ and $y$.
Differentiating $U^X_+ = U (V_-)^{-1}$ gives
\begin{equation}\label{duplus}
(U^X_+)^{-1} dU^X_+ = V_-(A \,dx + B\, dy)(V_-)^{-1}-dV_- (V_-)^{-1}.
\end{equation}
The left-hand side contains only positive powers of $\lambda$, while only the $dx$ term on
the right contains such powers, so it follows that $U^X_+$ depends on $x$ and $\lambda$ only.  A similar argument shows that $U^Y_-$ depends on $y$ and $\lambda$ only.
\end{proof}

For the rest of this section we will assume that the domain $D$ on which
$U(x,y)$ is defined contains the origin, and that $U(x,y)$ satisfies the following initial condition.  
\begin{equation}
U(0,0;\lambda)=I \;\;\forall{\lambda}. \label{initialized}
\end{equation}
Note that, by the uniqueness of the Birkhoff splitting, this implies that \\ $U^Y_-, U^X_+, V_-,V_+$ are also equal to the identity matrix when $x=y=0$.

\bigskip
With these conditions in hand, we can determine the differential equations for $U^X_+$ and $U^Y_-$.  First, the coefficients of non-positive powers of $\lambda$ on the right-hand side of \eqref{duplus}
must vanish, and hence
$$\dfrac{\di V_-}{\di y}=V_- B.$$
Letting $V_- = V_{-0}(x,y) + O(\lambda^{-1})$, we find that $\di V_{-0}/\di y =0$.  Taking coefficients of positive powers of $\lambda$ in \eqref{duplus}  now gives
\begin{equation}\label{duplusvzero}
(U^X_+)^{-1} \dfrac{\partial U^X_+}{\di x}  = \lambda V_{-0} A_1 (V_{-0})^{-1}.
\end{equation}
Next, we must determine $V_{-0}(x)$.  Restricting \eqref{duplus} to the line $y=0$ gives
\begin{equation}\label{duplusaxis}
(U^X_+)^{-1} \dfrac{\partial U^X_+}{\di x} = V_- A(V_-)^{-1}
-\left.\dfrac{\di V_-}{\di x} (V_-)^{-1}\right|_{y=0}.
\end{equation}
Again, the left-hand side contains only
positive powers of $\lambda$.
Taking the $\lambda^0$ coefficient in \eqref{duplusaxis} gives an expression for $\di V_{-0}/\di x$
which {\it a priori} involves products of $A_1$ with the coefficient of $\lambda^{-1}$ in $V_-$.
To eliminate such terms, we apply the following
\begin{lem}\label{sepps} Suppose $P(t) \in \Lambda^\pm SU(2)$, $Q(t)\in\Lambda^\mp SU(2)$ and
$R(t) \in \Lambda^\pm \su(2)$ satisfy
$$P^{-1}\dfrac{\di P}{\di t} = QRQ^{-1} - \dfrac{\di Q}{\di t} Q^{-1}$$
on some $t$-interval containing $t_0$, and $P(t_0) = Q(t_0) = I$.
Then $Q(t)$ has no $\lambda$-dependence.
\end{lem}
\begin{proof} For simplicity, take the upper sign in the hypotheses, the proof for the
lower sign being identical in form.  Let $S(t) \in \Lambda^+SU(2)$ satisfy
$$\dfrac{\di S}{\di t} = SR, \qquad S(t_0) = I.$$
Then
$$(SQ^{-1})^{-1} \dfrac{\di (SQ^{-1})}{\di t} = QRQ^{-1} - \dfrac{\di Q}{\di t} Q^{-1},$$
so that $SQ^{-1}$ satisfies the same differential equation, as a function of $t$, that $P$ does.
Since $S Q^{-1}$ also coincides with $P$ when $t=t_0$, it follows that $P = S Q^{-1}$ for all $t$-values
in the interval.  Hence $Q= P^{-1} S \in \Lambda^+ SU(2)$, and it follows that $Q$ has no $\lambda$-dependence.
\end{proof}

Using $t=x$, $P=U^X_+$, $Q=V_-\restr_{y=0}$ and $R=A\restr_{y=0}$ in the lemma, we conclude that
the restriction of $V_-$ to the $x$-axis has no $\lambda$-dependence.
Thus, we can replace $V_-$ with $V_{-0}$ in \eqref{duplusaxis}.
Taking the $\lambda^0$ coefficient in that equation and using \ref{su2sys} to get $A_0$ now gives
$$\dfrac{\di V_{-0}}{\di x}= V_{-0} A_0 = \frac{\ri}2 V_{-0}
\begin{bmatrix} -\phi_x(x,0) & 0\\ 0 & \phi_x(x,0) \end{bmatrix}.
$$
Using the initial condition $V_{-0}(0)=I$, we obtain
\begin{equation}\label{vzeroform}
V_{-0}(x) = \begin{bmatrix} e^{-\ri\alpha(x)/2} & 0 \\ 0 & e^{\ri \alpha(x)/2}\end{bmatrix}
\end{equation}
where we take $\alpha(x) =\phi(x,0)-\phi(0,0)$. Substituting in \eqref{duplusvzero} gives
\begin{equation}\label{duplusx}
(U^X_+)^{-1} \dfrac{\partial U^X_+}{\di x}
=\dfrac{\ri}2 \lambda \begin{bmatrix} 0 & e^{-\ri \alpha(x)} \\ e^{\ri \alpha(x)} & 0 \end{bmatrix} =: \eta_+^X.
\end{equation}

Similarly, differentiating $U^Y_- = U (V_+)^{-1}$ gives
$$
(U^Y_-)^{-1} dU^Y_- = V_+(A \,dx + B\, dy)(V_+)^{-1}-d V_+ (V_+)^{-1}.
$$
We restrict this equation to the $y$-axis, giving
\begin{equation}\label{duminusaxis}
(U^Y_-)^{-1}\dfrac{\di U^Y_-}{\di y} = V_+ B (V_+)^{-1}
-\dfrac{\di V_+}{\di y} (V_+)^{-1}.
\end{equation}
Using Lemma \ref{sepps} with $t=y$, $P=U_-$ and $Q=V_+$,
we conclude that the coefficients of all the positive powers of
$\lambda$ in $V_+$ vanish along the $y$-axis.
(However, the coefficient of $\lambda^0$ in $V_+$ will depend on both $x$ and $y$.)
Then, examining the $\lambda^0$ coefficient
in \eqref{duminusaxis} shows that, along the $y$-axis $V_+$ is constant and equal
to the identity matrix.  Then
\begin{equation}\label{duminusy}
(U^Y_-)^{-1}\dfrac{\di U^Y_-}{\di y}
=-\frac{\ri}{2\lambda}
\begin{bmatrix}
0 & e^{\ri\beta(y)} \\  e^{-\ri \beta(y)} & 0 \end{bmatrix}
=: \eta^Y_-
\end{equation} 
for $\beta(y) = \phi(0,y)$.

We will refer to $\eta^X_+(x, \lambda)$ in \eqref{duplusx} and
$\eta^Y_-(y,\lambda)$ in \eqref{duminusy} as a pair of {\em normalized potentials}, by analogy with holomorphic potentials that are determined by the loops associated to constant mean curvature surfaces.  
We remark the formulas expressing the potentials in terms of $\phi$,
which agree with those of Toda (\cite{T2}, equations (23) and (24)), 
are analogues of Wu's formula for CMC surfaces and their associated
harmonic maps \cite{Wu}.

\subsection{Normalized Potentials to Ps-Surface} 
Just as holomorphic potentials can be used to reconstruct constant mean curvature surfaces, we can use normalized potentials to reconstruct ps-surfaces.  To see how this works, suppose that $V_- = V_{-0} T_-$ for $T_-(x,y) \in \Lambda_*^- SU(2)$.  Then \eqref{usplit} implies that
$$(U^Y_-)^{-1} U^X_+ V_{-0} = V_+ (T_-)^{-1}.$$
The left-hand side is determined by the normalized potentials (i.e., by $\alpha(x)$ and $\beta(y)$), while the
right-hand side is a Birkhoff splitting (albeit, with the $\Lambda_*^-$ piece as the second factor).
Thus, given $\alpha(x)$ and $\beta(y)$ and applying the splitting, we determine $T_-(x,y)$, and thus
construct a loop satisfying \eqref{usplit} by setting $U=U_+^X V_{-0} T_-$.

More formally, we have the following

\begin{thm}\label{converse}  Let functions $\alpha(x)$ and $\beta(y)$ be defined on intervals $D_1,D_2 \subset \R$
containing zero, satisfying $\alpha(0)=0$.  Let $V_{-0}(x)$ be given by \eqref{vzeroform}, and
let $U^X_+(x), U^Y_-(y)$ satisfy \eqref{duplusx} and \eqref{duminusy}, respectively, with $U^X_+(0)=I$ and $U^Y_-(0)=I$.  
Let $D=D_1\times D_2$, 
and let
$$(U^Y_-)^{-1} U^X_+ V_{-0} = V_+ (T_-)^{-1}$$
be its Birkhoff decomposition, for $V_+(x,y) \in \Lambda^+ SU(2)$ and $T_-(x,y)\in \Lambda_*^-SU(2)$.
Then
$$U = U^X_+ V_{-0} T_- = U_-^Y V_+$$
satisfies \eqref{initialized} and also satisfies \eqref{su2sys} for a function $\phi(x,y)$ on $D$ such that
$$\phi(x,0) = \alpha(x)+\beta(0), \qquad \phi(0,y) =\beta(y).$$
\end{thm}
Of course, it follows that $\phi(x,y)$ satisfies the sine-Gordon equation \eqref{SGE}.
Moreover, the {\em Sym formula}:
\begin{equation}\label{sym}
f(x,y;\lambda) = \left.\dfrac{\di U}{\di \log \lambda} U^{-1} \right.
\end{equation}
gives a family of pseudospherical surfaces which, for each value of $\lambda$, have $U(x,y;\lambda)$ as extended frame.
(Note, however, that $f$ may fail to be an immersion at some points.)
\begin{proof}[Proof of Theorem \ref{converse}]
Let $V_- = V_{-0} T_-$.  Differentiating $U^Y_- V_+ = U^X_+ V_-$ with respect to $x$ gives
\begin{equation}\label{dwplusx}
(V_+)^{-1} \dfrac{\di V_+}{\di x}
= (V_-)^{-1} \dfrac{\di V_-}{\di x}
+ (V_-)^{-1}\left( (U^X_+)^{-1} \dfrac{\di U^X_+}{\di x}\right) V_-.
\end{equation}
The highest power of $\lambda$ on the right-hand side is $\lambda^1$
in the factor $(U^X_+)^{-1} \dfrac{\di U^X_+}{\di x}$, so
$U^{-1} \dfrac{\di U}{\di x}=(V_+)^{-1} \dfrac{\di V_+}{\di x}$ contains only $\lambda^0$ and $\lambda^1$.  Furthermore,
because of the twisting condition, the $\lambda^0$ term is diagonal and the $\lambda^1$ term
is off-diagonal.  Substituting \eqref{duplusx} into \eqref{dwplusx} gives
\begin{eqnarray}\label{duxderived}
U^{-1} \dfrac{\di U}{\di x}  & = &
\dfrac{\ri}2 \begin{bmatrix}-\kappa(x,y) & 0 \\ 0 & \kappa(x,y)\end{bmatrix}
+ (V_{-0})^{-1} \dfrac{\ri}2 \lambda \begin{bmatrix} 0 & e^{-\ri \alpha(x)} \\ e^{\ri \alpha(x)} & 0 \end{bmatrix} V_{-0} \nonumber \\
& = & \dfrac{\ri}2 \begin{bmatrix}-\kappa(x,y) & \lambda \\ \lambda & \kappa(x,y)\end{bmatrix}
\end{eqnarray}
for some function $\kappa(x,y)$.  Equation \eqref{dwplusx} implies
$$(U^X_+)^{-1} \dfrac{\di U^X_+}{\di x} = V_- \left((V_+)^{-1} \dfrac{\di V_+}{\di x}\right) (V_-)^{-1}
-\dfrac{\di V_-}{\di x} (V_-)^{-1}.$$
As with \eqref{duplusaxis}, applying Lemma \ref{sepps} to the last equation lets us
conclude that the restriction of $V_-$ to the $x$-axis has no negative
powers of $\lambda$.  Taking the $\lambda^0$ coefficient of the
restriction of \eqref{dwplusx} to the $x$-axis now gives
$$
\left.U^{-1} \dfrac{\di U}{\di x}\right|_{y=0} = (V_{-0})^{-1} \dfrac{\di V_{-0}}{\di x} +O(\lambda)
=\dfrac{\ri}2\begin{bmatrix} -\alpha'(x) & 0 \\ 0 & \alpha'(x)\end{bmatrix} + O(\lambda),
$$
which implies that $\kappa(x,0) = \alpha'(x)$.

Differentiating $U^X_+ V_- = U^Y_- V_+$ with respect to $y$ gives
\begin{equation}\label{vminusy}
(V_-)^{-1} \dfrac{\di V_-}{\di y} 
= (V_+)^{-1} \dfrac{\di V_+}{\di y}
+ (V_+)^{-1} \left( (U^Y_-)^{-1} \dfrac{\di U^Y_-}{\di y} \right) V_+.
\end{equation}

Because $\di V_{-0}/\di y=0$, only negative powers of $\lambda$ are present in
$U^{-1} \dfrac{\di U}{\di y} = (V_-)^{-1} \dfrac{\di V_-}{\di y}$.
Because the only power of $\lambda$ in $(U^Y_-)^{-1} \dfrac{\di U^Y_-}{\di y}$ is $\lambda^{-1}$,
then this is also the only power in $U^{-1} \dfrac{\di U}{\di y}$.
If we let $V_+ = V_{+0} T_+$ for $T_+ \in \Lambda^+_* SU(2)$, then $V_{+0}$ is diagonal because of
the twisting, and we can let
$$V_{+0}(x,y) = \begin{bmatrix} e^{\ri \psi(x,y)} & 0 \\ 0 & e^{-\ri\psi(x,y)} \end{bmatrix}$$
for some function $\psi$.
Then
\begin{eqnarray}\label{duyderived}
U^{-1} \dfrac{\di U}{\di y} &=& (V_{+0})^{-1} \left((U^Y_-)^{-1} \dfrac{\di U^Y_-}{\di y} \right) V_{+0} \nonumber \\ 
& = & 
-\dfrac{\ri}{2\lambda} \begin{bmatrix} 0 & e^{\ri (\beta(y)-2\psi(x,y)} \\ e^{-\ri(\beta(y)-2\psi(x,y))}&0\end{bmatrix}.
\end{eqnarray}
In order to be consistent with the $y$-derivative in \eqref{su2sys}, we let
$$\phi(x,y) = \beta(y)-2\psi(x,y).$$
Then the compatibility condition between \eqref{duxderived} and \eqref{duyderived}
implies that $\kappa=\phi_x$ and $\phi$ satisfies the sine-Gordon equation.

Lastly, we want to show that $\phi(0,y)=\beta(y)$, 
which is equivalent to $V_{+0}(0,y)=I$.
Letting $B$ stand for the right-hand side in \eqref{duyderived}, then
equation \eqref{vminusy} can be rearranged to give
$$(U^Y_-)^{-1} \dfrac{\di U^Y_-}{\di y} = V_+ B (V_+)^{-1} - \dfrac{\di V_+}{\di y}(V_+)^{-1}.$$
As with \eqref{duminusaxis},
we restrict this equation to the $y$-axis, and apply Lemma \ref{sepps} to conclude that
$V_+(0,y)$ has no positive powers of $\lambda$.  We can replace $V_+$ with $V_{+0}$ in
the restriction of \eqref{duyderived} to the $y$-axis, and taking the $\lambda^0$ coefficient
in that equation now shows that $V_{+0}(0,y)=I$.
\end{proof}

\subsection{Uniqueness and Differentiability} \label{one2one} 
\def\f{\mathsf f}
Let $\eta_f$ denote the unique normalized potential associated to an oriented  \linebreak ps-immersion $f$ by the algorithm of 
\S\S\ref{ps2Darboux}--\ref{extended2normalized}.
(We use $\eta$ as an abbreviation for
a pair $(\eta^Y_-, \eta^X_+)$ of normalized potentials.)
Then if $f$ and $\hat{f}$ are two ps-immersions that differ by a rigid motion,
$\eta_f = \eta_{\hat{f}}$.

Let $\f[\eta]$ denote the ps-immersion associated to a
potential pair $\eta$ by the construction of Theorem \ref{converse}
and the Sym formula \eqref{sym} with $\lambda=1$.
Suppose $f(x,y)$ is a ps-immersion with angle function $\phi(x,y)$.
Then by Theorem \ref{converse} the angle function $\widetilde\phi$ for
$\f[\eta_{f}]$ agrees with $\phi$ along the $x$- and $y$-axes.
It follows by a theorem of Bianchi that $\widetilde\phi = \phi$,
and by a theorem of Enneper (cf. \cite{Tthesis} Thm. 2.2.1 or \cite{T2} Cor. 2) that $f$ and $\f[\eta_{f}]$ differ by a rigid motion.
Thus, there is a 1:1 relation, up to rigid motion, between (local) ps-immersions and normalized potentials.

So far, we have assumed that all objects under consideration are smooth, but
we can be more precise about how degrees of differentiability behave
under these correspondences.
Let $f$ be a ps-immersion which is $C^n$ for $n \geq 2$.
Then $\phi$ is $C^{n-1}$, and therefore the potential pair 
given by \eqref{duplusx} and \eqref{duminusy} is $C^{n-1}$ in $x,y$ and analytic in $\lambda\in \C^*$.

Conversely, consider some potential pair $\eta$ which is $C^{n-1}$ in $x,y$ for $n \geq 3$ and analytic in $\lambda\in \C^*$.  
Then by solving \eqref{duplusx} and \eqref{duminusy}
we obtain $U_+^X, U_-^Y$ which are $C^{n}$ in $x,y$ and analytic in $\lambda$.  The next step in our construction is the Birkhoff splitting of 
$(U_-^Y)^{-1} U_+^X$ as $V_+ T_-^{-1}$.  
Since the splitting is analytic \cite{G},
$V_+, T_-$ are $C^{n}$ in $x,y$ and analytic in $\lambda$.  
Therefore $U(x,y,\lambda)$ is $C^{n}$ in $x,y$ and analytic in $\lambda$.  
Now we use the Sym formula \eqref{sym}, which differentiates with respect to $\log\lambda$, to obtain a mapping $f$ which is $C^{n}$ in $x,y$ 
and analytic in $\lambda$.

For each fixed value of $\lambda$, $f$ will be an immersion except
at points where $\sin\phi(x,y)=0$.  (At such points,
the surface could have singularities, either weakly regular
points \cite{MS} or cone points, as we will see in Part II.)
To see why these are the only points where $f$ could fail to be an immersion,
use the Sym formula and \eqref{su2sys} to compute
$f_x = U V' U^{-1}$  and $f_y = U W' U^{-1}$, where 
$$V = \frac{\ri}2 
\begin{bmatrix} -\phi_x & \lambda \\ \lambda & \phi_x \end{bmatrix}, \qquad W=
 -\frac{\ri}{2\lambda} 
\begin{bmatrix}
0 & e^{\ri\phi} \\  e^{-\ri \phi} & 0 \end{bmatrix}
$$
and prime denotes differentiation with respect to $\log\lambda$.
From the explicit expression for $V'$ and $W'$ it is clear that $f_x$ and $f_y$ are linearly dependent if and only if $\sin \phi = 0$.
   

\subsection{Generalized Potentials}\label{generalizedsection}
In \S\S\ref{ps2Darboux}--\ref{extended2normalized} we have reviewed the procedure which associates with every ps-surface a pair of potentials 
$\eta = (\eta^X_+(x,\lambda),\eta^Y_-(y,\lambda)$ given
by \eqref{duplusx} and \eqref{duminusy}.  In these
potentials the dependence on $\lambda$ is quite simple:
$\eta^X_+ = \lambda \xi^X_+$ and $\eta^Y_-= \lambda^{-1} \xi^Y_-$
for a pair of off-diagonal skew-Hermitian $2 \times 2$
matrices $\xi^X_+$ and $\xi^Y_-$ depending only on $x$ and $y$ respectively.

It is also possible to construct ps-surfaces from potentials which are more general, in the sense that more powers of $\lambda$ are involved.  To indicate that more than one power of $\lambda$ may be involved we will use the notation
\begin{equation}\label{genexpand}
\begin{aligned}
\eta^X_{\sharp}(x,\lambda) &= \lambda \xi_{1}^X(x) + \lambda^{0} \xi_{0}^X(x) + \mathcal{O} (\lambda^{-1}),\\
\eta^Y_\flat(y,\lambda) &= \lambda^{-1} \xi_1 ^Y(y) + \lambda^{0} \xi_{0}^Y(y) + \mathcal{O }(\lambda),\\
\end{aligned}\end{equation}
Note in particular that $\eta^X_{\sharp}$ is not necessarily in $\Lambda^+SU(2)$ and $\eta^Y_\flat$ is not necessarily in $\Lambda^-SU(2)$.  

\begin{defn}
We say that $\su(2)$-valued loops $\eta^X_\sharp(x)$ and $\eta^Y_\flat(y)$ are
{\em generalized potentials for a pseudospherical surface} if the extended frame $U$ of the surface satisfies splittings of the form
\begin{equation}
U = G^X_\sharp L_- = G^Y_\flat L_+, \label{generalsplit}
\end{equation}
for some $L_\pm(x,y) \in \Lambda^\pm SU(2)$ and $SU(2)$-valued matrices 
$G^X_\sharp(x,\lambda), G^Y_\flat(y,\lambda)$ which have $\eta^X_\sharp,\eta^Y_\flat$ as their respective Maurer-Cartan matrices, i.e.,
\begin{equation}\label{MCgees}
\dfrac{\di G^X_\sharp}{\di x} = G^X_\sharp \eta^X_\sharp, \qquad
\dfrac{\di G^Y_\flat}{\di y} = G^Y_\flat \eta^Y_\flat.
\end{equation}
\end{defn}
Note that we do not impose the initial condition \eqref{initialized} on $U$ or the other factors in \eqref{generalsplit}.

\begin{thm}\label{surfpot}
Let $D_1, D_2$ be intervals on the real line and
let $\eta^X_\sharp(x)$ and $\eta^Y_\flat(y)$ be two loops in $\mathfrak{sl}(2,\C)$, defined for $x\in D_1$ and $y\in D_2$ respectively, 
which are smooth in $x$ and $y$, analytic in $\lambda \in \C^*$,
and which have expansions of the form \eqref{genexpand} with nonvanishing $\xi^X_1$ and $\xi^Y_1$.  We also assume that these loops are $\su(2)$-valued
when $\lambda$ is real, and that they satisfy the twisting condition
given in \eqref{algloopdef}.
Then $\eta^X_\sharp,\eta^Y_\flat$ are generalized potentials
for a pseudospherical immersion defined for $(x,y) \in D = D_1\times D_2$.
\end{thm}
Note that the twisting condition implies that the coefficients of 
even powers of $\lambda$ take value in $\h_0$, the diagonal subalgebra of
$\su(2)$, while the coefficients of odd powers take value in $\h_1$,
the subspace of off-diagonal matrices.

\begin{proof}
Let $G^X_\sharp, G^Y_\flat$ be as in \eqref{MCgees}.
Suppose that
$$\xi^X_1 = \dfrac{\ri}2 a(x)  \begin{pmatrix} 
0 & e^{-\ri\alpha(x)} \\
e^{\ri\alpha(x)} & 0 \end{pmatrix}$$
for a nonzero real function $a(x)$, and let
$$T^X = \begin{pmatrix}
e^{-\ri\alpha(x)/2} & 0 \\
0 & e^{\ri\alpha(x)/2}
\end{pmatrix}.
$$
Construct the Birkhoff splitting
\begin{equation}\label{splitgees}
(G^Y_\flat)^{-1} G^X_\sharp T^X = L_+ (L_-)^{-1},
\end{equation}
where $L_+(x,y) \in \Lambda^+ SU(2)$ and $L_-(x,y) \in \Lambda^-_*SU(2)$,
and let
$$U = G^X_\sharp T^X L_- = G^Y_\flat L_+.$$
We compute
$$U^{-1} \dfrac{\di U }{\di y} = (L_-)^{-1}\dfrac{\di L_-}{\di y}
= (L_+)^{-1} \left(\eta^Y_\flat L_+ + \dfrac{\di L_+}{\di y}\right).$$
The right-hand side has at most one negative power of $\lambda$, while the middle member
is in $\Lambda^-_* \su(2)$.  Thus, 
if we write $L_+= L_{+0}(x,y)+\mathcal{O}(\lambda)$, then 
$U^{-1} \dfrac{\di U }{\di y} = (L_{+0})^{-1} \xi^Y_1 L_{+0}$
and is $\h_1$-valued. Similarly, we compute
$$U^{-1} \dfrac{\di U }{\di x} = (L_+)^{-1}\dfrac{\di L_+}{\di x}
= (\widetilde L_-)^{-1} \left(\eta^X_\sharp \widetilde L_- + \dfrac{\di \widetilde L_-}{\di x}\right),$$
where $\widetilde L_- = T^X L_-$. 
The right-hand side has at most one positive power of $\lambda$, while the middle
member has only non-negative powers of $\lambda$, so we have
$$U^{-1} \dfrac{\di U }{\di x} = A_0 + \lambda A_1,$$
where $A_0$ depends on $L_+$, but 
$$A_1 = (T^X)^{-1} \xi_1^X T^X = \dfrac{\ri}2 a(x) \begin{pmatrix}0 & 1 \\ 1 & 0 \end{pmatrix}.$$

Suppose that 
$$\xi^Y_1 = \dfrac{\ri}{2} \begin{pmatrix} 0 & \rho(y) \\ \overline{\rho(y)} & 0 \end{pmatrix},
\qquad L_{+0} = \begin{pmatrix} e^{\ri \psi(x,y)} & 0 \\ 0 & e^{-\ri \psi(x,y)} \end{pmatrix}.
$$
Then $U$ satisfies the linear system
\begin{equation}\label{constructUlinsys}
\dfrac{\di U}{\di x} = \dfrac{\ri}2 a(x) U 
\begin{pmatrix} -\kappa(x,y) & \lambda \\  
\lambda & \kappa(x,y) \end{pmatrix}, \qquad
\dfrac{\di U}{\di y} = \dfrac{\ri}{2\lambda} U
\begin{pmatrix} 0 & \rho e^{-2\ri \psi} \\
\overline{\rho} e^{2\ri \psi} & 0 \end{pmatrix},
\end{equation}
for some function $\kappa$.  Set $\rho(y) = -b(y) e^{\ri\beta(y)}$ for a real-valued function $b(y)>0$.  Then the second 
equation in \eqref{constructUlinsys} agrees with the second equation in \eqref{su2sys}, with 
the speed $b(y)$ inserted, for
$$\phi(x,y) = \beta(y) - 2\psi(x,y).$$
Then the compatibility condition for \eqref{constructUlinsys} implies that
$\kappa = (a(x))^{-1} \di \phi /\di x$.  Note that $\phi$ satisfies a version 
of the sine-Gordon equation:
$$\phi_{xy} = a(x) b(y) \sin\phi.$$

The existence of a pseudospherical surface with extended frame $U$ now follows by integration.
\end{proof}

\begin{thm}\label{gaugetheorem}
Assume $\eta = (\eta^X_{\sharp}(x), \eta^Y_{\flat}(y))$ is a generalized potential pair of some ps-surface, satisfying the hypotheses of Theorem \ref{surfpot}.
Given any
$Q^X_-(x,\lambda)$ \linebreak $\in \Lambda^-SU(2)$ and $Q^Y_+(y,\lambda) \in \Lambda^+SU(2)$
defined for $x\in D_1$ and $y\in D_2$, let
\begin{equation}\label{gaugeit}
\begin{aligned}
\widetilde{\eta}^X_{\sharp} &= 
(Q_-^X)^{-1} \eta^X_{\sharp} Q_-^X + (Q_-^X)^{-1} \frac{d}{dx} Q_-^X,\\
\widetilde{\eta}^Y_{\flat} &= 
(Q_+^Y)^{-1} \eta^Y_{\flat} Q_+^Y + (Q_+^Y)^{-1} \frac{d}{dy} Q_+^Y.
\end{aligned}
\end{equation}
Then $\widetilde{\eta}=(\widetilde{\eta}^X_{\sharp},\widetilde{\eta}^Y_{\flat})$ is a generalized potential for the same surface.
Moreover, any two generalized potentials for the same surface are related by some gauge transformation of the form \eqref{gaugeit}.
\end{thm}

\begin{proof}
To prove the first assertion, note that $\widetilde G^X_\sharp = G^X_\sharp Q^X_-$
and $\widetilde G^Y_\flat = G^Y_\flat Q^Y_+$ satisfy \eqref{MCgees} with 
$\eta$ replaced by $\widetilde{\eta}$.  Then the corresponding
splitting in \eqref{splitgees} is satisfied by
$\widetilde L^+ = (Q_+^Y)^{-1} L^+$ and $\widetilde L_- = (T^X)^{-1}Q^X_{-0} (Q^X_-)^{-1} T^X L_-$, where $Q^X_{-0}$ denotes the term in $Q^X_-$ of
order zero in $\lambda$.  Then $\widetilde U = \widetilde G^Y_\flat \widetilde L_+ = U$.

For the proof of the second assertion we recall from \S\ref{one2one} that normalized potentials are 
in 1-1 relation to local ps-immersions up to rigid motions.  Therefore, if two potentials  $\widetilde{\eta}$ and  $\eta$ induce the same surface, then, after fixing the frame at some basepoint, 
we can assume that they induce the same extended frame.  Comparing now the defining equations for the extended frame we obtain (for example)
\[U=G^X_\sharp L_- = \widetilde{G}^X_\sharp \widetilde{L}_-\]
where
\[\frac{\partial G^X_\sharp}{\partial x} = G^X_\sharp \eta^X_\sharp, \qquad
\frac{\partial \widetilde{G}^X_\sharp}{\partial x} = \widetilde{G}^X_\sharp \widetilde{\eta}^X_\sharp.\]
Letting $Q^X_- = L_-\widetilde{L}^{-1}_-\restr_{y=0}$, we have
$\widetilde{G}^X_\sharp = G^X_\sharp Q^X_-,$ and thus
$$\widetilde{\eta}^X_\sharp 
=  (\widetilde G^X_\sharp)^{-1} \dfrac{\di \widetilde G^X_\sharp}{\di x}
=
(Q^X_-)^{-1} \eta^X_\sharp Q^X_- + 
(Q^X_-)^{-1} \frac{\partial Q^X_-}{\partial x}.  \qedhere 
$$ 
\end{proof}

The result above implies a somewhat surprising but useful

\begin{cor}
Assume that $U(x,y;\lambda)$ is the extended frame of some ps-surface.
Then  
$$\eta^X_\sharp(x,\lambda) := \left.U^{-1}\dfrac{\di U}{\di x}\right|_{y=x},
\qquad \eta^Y_\flat(y,\lambda) := \left.U^{-1}\dfrac{\di U}{\di y}\right|_{x=y}
$$
are generalized potentials for the same surface.
\end{cor}
\begin{proof}
We can show that $(\eta^X_\sharp,\eta^Y_\flat)$ differ from the normalized
potential $(\eta^X_+, \eta^Y_-)$ of $U$ by a gauge transformation \eqref{gaugeit}.
For example, from \eqref{usplit} and \eqref{duplusx} we have
$$U^{-1}\dfrac{\di U}{\di x} = (V_-)^{-1} \eta^X_+ V_- + (V_-) \dfrac{\di V_-}{\di x},$$
where $V_-$ depends on $x$ and $y$.  Letting $Q^X_- = V_-\restr_{y=x}$, we have
$$\eta_\sharp^X = (Q_-^X)^{-1} \eta_+^X Q_-^X + (Q_-^X)^{-1} \frac{\partial Q_-^X}{\partial x}. \qedhere $$
\end{proof}

\section{Symmetric Ps-Surfaces} \label{sec:startsym}
\def\su{\mathfrak{su}}
\def\w{\omega}
\def\wL{\omega^\lambda}
\def\FL{F^\lambda}

\subsection{Symmetric Ps-Surfaces and Frames}
Let $f:D \rightarrow \R^3$ be a ps-immersion, where $D \subset \R^2$ is an open set.  We assume that $f$ is nondegenerate at each point in $D$, and that the $x$- and $y$-coordinate curves are asymptotic lines on the surface $f(D)$
(not necessarily of unit speed).  By convention, we associate an oriented orthonormal frame $F(x,y)$ to $f$ in a unique way, so that the first column of $F$ is the unit vector in the direction of $f_x$, and the third column is the unit vector in the direction of the surface normal $\mathbf{n} = f_x \times f_y$.

The most natural notion of a symmetry of an immersion $f$
seems to be that there exists some rigid motion $R$ such that
$R \circ f(D) = f(D)$.  
Around any 
point $p_0$ in this subset there exists an open set $U\subset D$ and map $\gamma :U \to D$ such that $f \circ \gamma (p) = R \circ f(p)$ and $\gamma$ is a diffeomorphism onto its image.
In many cases (and under natural assumptions, like completeness of the pullback metric on $D$) such a $\gamma$ exists globally.
We therefore will use from now on the \textbf{basic assumption} that there is a rigid motion $R:\R^3 \to \R^3$ and a diffeomorphism
$\gamma:D \to D$ such that
\begin{equation}\label{phigammaR}
f \circ \gamma = R \circ f.
\end{equation}

\begin{prop}
Let $R' \in O(3)$ be the linear part of $R$.  Then there is a matrix $K(x,y) \in O(3)$,
with $\det K = \det R'$, such
that
\begin{equation}\label{FgammaR}
F \circ \gamma = R'\,F \, K.
\end{equation}
Moreover, the Maurer-Cartan form $K^{-1} dK$ takes value in the subalgebra
$$\k_0 = \left\{ \begin{bmatrix} 0 & a & 0 \\ -a & 0 & 0 \\ 0 & 0 & 0\end{bmatrix}\right\}
\subset \mathfrak{so}(3).$$
\end{prop}

In fact, the matrix $K(x,y)$ can be calculated as follows: Let $\epsilon = \pm1$
be such that $R'{\mathbf n} = \epsilon (\mathbf n \circ \gamma)$, let $J$ be the Jacobian of $\gamma$, and let $Z$ be the $2 \times 2$ upper triangular matrix such that
$$[f_x, f_y] = F \begin{bmatrix}Z \\ \begin{array}{cc} 0 & 0 \end{array}\end{bmatrix}.$$
Then
$$
K(x,y) = \begin{bmatrix} Z J^{-1} (Z \circ \gamma)^{-1} & 0 \\ 0 & \epsilon \end{bmatrix}.
$$

From now on, assume that $R'$ and $K(x,y)$ are $SO(3)$-valued, i.e., the rigid motion is proper.

\subsection{Extended Frames and Monodromy}
Note that $\gamma$ preserves asymptotic coordinates.  For, it follows from
\eqref{phigammaR} that 
$$(f \circ \gamma)_x = R' f_x,\quad (f \circ \gamma)_y = R' f_y,\quad
(f \circ \gamma)_{xx} = R' f_{xx}.$$  
Hence $(f \circ \gamma)_x$ is an asymptotic vector if and only if $\det[R' f_x,R' f_y,R' f_{xx}]=\det R' \det[f_x, f_y, f_{xx}]=0$, i.e., $f_x$ is an asymptotic vector.  The same is true for $(f \circ \gamma)_y$.  Thus,
$\gamma$ either preserves or switches the $(x,y)$ coordinates.

\begin{prop}\label{chiprop}  Let $F^\lambda$ be the extended frame for the given ps-surface.
If $\gamma$ takes $x$-coordinate curves to $x$-coordinate curves
and $y$-coordinate curves to $y$-coordinate curves, then there is an $SO(3)$-valued
function $\chi(\lambda)$ such that
\begin{equation}\label{Fgammachi}
\FL \circ \gamma = \chi\, \FL\, K,
\end{equation}
and moreover $K$ depends only on $x$.
If $\gamma$ switches the $x$- and $y$-coordinate curves, then we have
\begin{equation}\label{Fgammachirev}
F^\lambda \circ \gamma = \chi\, F^{1/\lambda} \, K.
\end{equation}
\end{prop}
\begin{proof} Let $\wL = (\FL)^{-1} d\FL,$
where $\lambda$ is treated as a constant.
From \eqref{ftlsys} we know that
$$\wL = (\alpha_0 + \alpha_1\lambda) dx + \beta_1 dy,$$
where $\alpha_0$ takes value in $\h_0$ and $\alpha_1,\beta_1$ in $\h_1$.

Differentiating each side of \eqref{FgammaR} gives
$$\gamma^* \w = K^{-1} \w K + K^{-1} dK,$$
where $\w$ simply denotes the value of $\wL$ when $\lambda=1$.
Taking the $\h_0$ and $\h_1$ parts of each side in this equation, we have
\begin{equation}\label{dtwineparts}
\begin{aligned}
\gamma^*(\alpha_0\,dx) &= K^{-1} \alpha_0 K\,dx + K^{-1} dK,\\
\gamma^*(\alpha_1\,dx + \beta_1 \,dy) &=K^{-1}( \alpha_1 \,dx + \beta_1 \,dy) K.
\end{aligned}
\end{equation}

\noindent
1. Assume that $\gamma$ preserves the set of $x$-coordinate curves.  Then
$\gamma^* dx$ is a multiple of $dx$, and the
second line in \eqref{dtwineparts} implies that
\begin{align*}
\gamma^*(\alpha_1\,dx) &= K^{-1}\alpha_1K \,dx,\\
\gamma^*(\beta_1 \,dy) &= K^{-1}\beta_1 K \,dy.
\end{align*}
Multiplying these by powers of $\lambda$ and combining with the first line in
\eqref{dtwineparts} gives
$$\gamma^*\wL = K^{-1} \wL K + K^{-1} dK.$$
The left-hand side is the Maurer-Cartan form of $\FL \circ \gamma$, while
the right-hand side is the Maurer-Cartan form of $\FL K$.  Then \eqref{Fgammachi}
follows by a standard theorem for maps into Lie groups (see, e.g., Theorem 10.18 in \cite{S}).
Note that we must assume that the domain $D$ is connected to use this.  The
$x$-dependence of $K$ follows from the first line of \eqref{dtwineparts}.

\medskip
\noindent
2. Instead, assume that $\gamma$ exchanges the $x$- and $y$-coordinate curves.  Then
the second line in \eqref{dtwineparts} implies that
\begin{align*}
\gamma^*(\alpha_1\,dx) &= K^{-1}\beta_1K \,dy,\\
\gamma^*(\beta_1 \,dy) &= K^{-1}\alpha_1 K \,dx.
\end{align*}
Multiplying these by powers of $\lambda$ and combining with the first line in
\eqref{dtwineparts} gives
$$\gamma^*\wL = K^{-1} \w^{1/\lambda} K + K^{-1} dK.$$
Now \eqref{Fgammachirev} follows from this by the same argument as above.
\end{proof}

We lift the equation \eqref{Fgammachi} up to $SU(2)$,
whereupon $K^{-1} dK$ takes value in the subalgebra $\h_0$ of diagonal
matrices in $\su(2)$.
Then the lift $U$ of $F^\lambda$ to $SU(2)$ satisfies
\begin{equation}\label{su2sysstretched}
\dfrac{\di U}{\di x} = \frac{\ri}2 U
\begin{bmatrix} -\phi_x & a(x)\lambda \\ a(x)\lambda & \phi_x \end{bmatrix} , \qquad
\dfrac{\di U}{\di y} = -\frac{\ri b(y)}{2\lambda} U
\begin{bmatrix}
0 & e^{\ri \phi} \\  e^{-\ri \phi} & 0 \end{bmatrix},
\end{equation}
where $a = |\di f/\di x|$ and $b = |\di f/\di y|$ are the speeds in the
$x$- and $y$-directions.  We will write the lifted version of \eqref{Fgammachi} as
\begin{equation}\label{Ugammachi}
U \circ \gamma = \chi\, U K
\end{equation}

\begin{prop}\label{SeppsPropA}
Assume that $SU(2)$-valued functions $U(x,y;\lambda)$, $K(x,y)$ and
$\chi(\lambda)$ satisfy \eqref{su2sysstretched} and \eqref{Ugammachi} on $D$.
Then $\gamma$ takes $x$-coordinate curves to $x$-coordinate curves and
$y$-coordinate curves to $y$-coordinate curves.
\end{prop}
\begin{proof}  We apply a Birkhoff splitting to $U$, yielding
\begin{equation}\label{Ubirkhoffsplit}
U = U^X_+ V_- = U^Y_- V_+,
\end{equation}
where, by Prop. \ref{onlyXonlyY}, $U^X_+$ has no $y$-dependence and
$U^Y_-$ has no $x$-dependence.  Substituting the last splitting into
\eqref{Ugammachi} and removing $V_+$ factors from the left-hand side gives 
$$U^Y_- \circ \gamma = \chi U^Y_- W_+,\qquad 
\text{where } W_+ = V_+ K (V_+ \circ \gamma)^{-1}.
$$
Let $(\gamma_1(x,y), \gamma_2(x,y))$ be the components of $\gamma$.
Differentiating the last equation with respect to $x$ 
and canceling $U^Y_- \circ \gamma$ from each side gives
$$\dfrac{\di \gamma_2}{\di x}\, (\eta_- \circ \gamma_2)  = W_+^{-1} \dfrac{\di W_+}{\di x},$$
where $\eta_-$ is as in \eqref{duminusy}.  Thus, the left-hand side contains only negative
powers of $\lambda$, while the right-hand side contains no negative powers of $\lambda$.
Because $\eta_-$ does not vanish, we must have $\di\gamma_2/\di x=0$
(and also $\di W_+/\di x =0$).  A similar argument shows that $\di\gamma_1/\di y=0$.
\end{proof}

\subsection{Extended Frames and Equivariant Potentials}
In this section we consider the implications of the symmetry assumption
for potentials derived from splitting the extended frame.  As in \S\ref{generalizedsection} we will consider more general splittings of the form
\begin{equation}\label{Ugeneralsplit}
U = G^Y_\flat L_+ = G^X_\sharp L_-,
\end{equation}
where $G^X_\sharp$ and $G^Y_\flat$ are assumed to have the property that $\di G^X_\sharp / \di y=0$, $\di G^Y_\flat/ \di x=0$, and $L_\pm \in \Lambda^\pm SU(2)$.
As in \S\ref{generalizedsection}, let $\eta^X_\sharp$ and $\eta^Y_\flat$ be the Maurer-Cartan matrices of $G^X_\sharp$ and $G^Y_\flat$ respectively, 
defined by \eqref{MCgees}, which are assumed to have expansions of the form \eqref{genexpand}.

\begin{lem}\label{Getatransform}
Let $U(x,y;\lambda)$, $K(x,y)$ and
$\chi(\lambda)$ be as in Prop. \ref{SeppsPropA} and assume
we have splittings as in \eqref{Ugeneralsplit}.
Then these factors transform under
$\gamma$ according to
\begin{align}
\label{Gminusgamma}
G^Y_\flat \circ \gamma &= \chi\, G^Y_\flat W^Y_+, \\
\intertext{and}
\label{Gplusgamma}
G^X_\sharp \circ \gamma &= \chi\, G^X_\sharp  W^X_-,
\end{align}
where $W^X_- = L_- K (L_- \circ \gamma)^{-1}$ takes value in
$\Lambda^- SU(2)$ and has no $y$-dependence,
 and $W^Y_+ = L_+ K (L_+ \circ \gamma)^{-1}$ takes value in $\Lambda^+ SU(2)$ and has no $x$-dependence.
Then the generalized potentials of $G^X_\sharp$, $G^Y_\flat$ transform according to
\begin{align}\label{EtaTransformX}
(\eta^X_\sharp \circ \gamma) \dfrac{d \gamma_1}{dx} &= (W^X_-)^{-1} \eta^X_\sharp W^X_- + (W^X_-)^{-1} \dfrac{d W^X_-}{dx},\\ \label{EtaTransformY}
(\eta^Y_\flat \circ \gamma) \dfrac{d \gamma_2}{dy} &= (W^Y_+)^{-1} \eta^Y_\flat W^Y_+ + (W^Y_+)^{-1} \dfrac{d W^Y_+}{dy},
\end{align}
where $\gamma_1, \gamma_2$ are the components of $\gamma$.
\end{lem}
\begin{proof} Substituting the splitting $U=G^Y_\flat L_+$ into \eqref{Ugammachi}
gives \eqref{Gminusgamma}, and we similarly deduce \eqref{Gplusgamma}.  
By Prop. \ref{SeppsPropA}, $\gamma$ preserves
$x$-coordinate and $y$-coordinate curves separately.
Differentiating
\eqref{Gminusgamma} and \eqref{Gplusgamma} then yields the last two equations.
\end{proof}

\begin{prop}\label{SeppsPropB}
Let $\gamma:D\to D$ separately preserve $x$- and $y$-coordinate lines.
Assume that matrices $G^X_\sharp(x,\lambda)$, $G^Y_\flat(y,\lambda)$ satisfy \eqref{Gminusgamma} and \eqref{Gplusgamma} for some $W^Y_+(y,\lambda)$ and $W^X_-(x,\lambda)$ in $\Lambda^+SU(2)$ and $\Lambda_- SU(2)$ respectively, and the same $\chi(\lambda)\in SU(2)$.  Then there exist $U, L_+, L_-$ satisfying \eqref{Ugeneralsplit}
and a $K(x,y) \in SU(2)$ such that \eqref{Ugammachi} is satisfied for every $\lambda$.
\end{prop}
\begin{proof}
Apply a Birkhoff splitting to the product $(G^Y_\flat)^{-1} G^X_\sharp$, to find matrices $L_\pm(x,y,\lambda) \in \Lambda^\pm SU(2)$ such that
$$L_+ (L_-)^{-1} = (G^Y_\flat)^{-1} G^X_\sharp.$$
Composing with $\gamma$ and applying the intertwining relations \eqref{Gminusgamma}
and \eqref{Gplusgamma} yields
$$(L_+ \circ \gamma) (L_-\circ \gamma)^{-1} = W^X_+ L_+ (W^Y_- L_-)^{-1}.$$
Because we are not imposing that any factor lie in  $\Lambda^\pm_* SU(2)$, the splitting factors are not unique;  however the minus and plus factors on each side
must agree up to a multiple that is independent of $\lambda$.  Thus, we have
$$L_+ \circ \gamma = W^Y_+ L_+ K(x,y),
\qquad L_- \circ \gamma = W^X_- L_- K(x,y)$$
for some $K(x,y) \in SU(2)$.  From this it follows that $U = G^Y_\flat L_+=G^X_\sharp L_-$ satisfies \eqref{Ugammachi}.  (Moreover, because all loops are assumed to be twisted by the automorphism of $\su(2)$ that preserves $\h_0$, it follows that $K^{-1} dK$ takes value in $\h_0$.)
\end{proof}

Our starting point in constructing symmetric ps-surfaces will be to write down
potentials satisfying \eqref{EtaTransformX} and \eqref{EtaTransformY}. 
Note that these conditions mean that the pullbacks
under $\gamma$ of the 1-forms $\eta^X_\sharp dx$ and $\eta^Y_\flat dy$ are gauge-equivalent to themselves, under the gauges $W^X_-$ and $W^Y_+$ respectively.
It is easiest to satisfy these if the gauge matrices $W^Y_+$ and $W^X_-$ are constant, as in the following. 

\subsection{Important Example} \label{sec:example}

The general theory presented in this paper is a powerful tool for producing new classes of examples of ps-surfaces.  The ps-cone points arising naturally in several of these examples help to clarify both the theory of discrete ps-surfaces and the theory of Lorentz umbilic points.  In ``Part II" we present these and other special cases of symmetry.  Included among the special cases are all known examples and many new examples.  

Here we will not derive but simply give the input data for an important example with discrete rotational symmetry.  Let $C(x) = (x-\ri)/(x+\ri)$ denote
the Cayley transform, which maps the real line onto the unit circle.
We choose
$$\eta^X_\sharp = (\lambda+\lambda^{-1}) 
\begin{bmatrix} 0 & p(x) \\ -\overline{p(x)} & 0 \end{bmatrix},
\quad 
\eta^Y_\flat = (\lambda+\lambda^{-1}) 
\begin{bmatrix} 0 & p(y) \\ -\overline{p(y)} & 0 \end{bmatrix},
$$
where the function $p$ which appears in both matrices is defined
by 
$$p(t) = \dfrac{d}{dt}\left( \dfrac{Q(w)}w\right),\qquad 
\text{where }
Q(w)=w^3+w^{-3} \text{ and } w=C(t).$$
Then $(\eta^X_\sharp, \eta^Y_\flat)$ satisfies
the conditions \eqref{EtaTransformX}, \eqref{EtaTransformY} for
$$W^X_- = W^Y_+ = \begin{bmatrix} e^{\ri \pi/3} & 0 \\ 0 & e^{-\ri \pi/3}\end{bmatrix}$$
and maps $(\gamma_1, \gamma_2)$ of the $x$ and $y$ axes 
that each correspond under the Cayley transform
to rotation by $2\pi/3$ on the unit circle.


This choice of potential yields the ps-surface shown (from different viewpoints) in Figures 1 and 2.  The surface is symmetric under a rotation (through an angle of approximately $\pi/3$) around an axis which is perpendicular to the page.  The figures show only a portion of the ``bottom half'' of the surface, the  ``top half'' being given by  reflecting through the plane perpendicular to the axis.  This axis passes through a pseudospherical cone point.  This point itself is the image of a line in the domain and could be thought of as a degenerate curvature line.  It is noteworthy that every asymptotic curve (in both families) goes through the cone point. 

\begin{figure}
\includegraphics[scale=1]{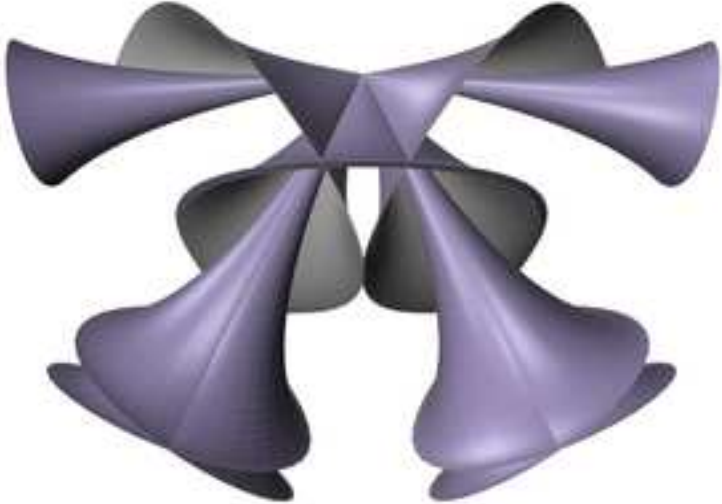}
\caption{Generalized Amsler Surface}
\label{fig:dis2}
\end{figure}

\begin{figure}
\includegraphics[scale=1]{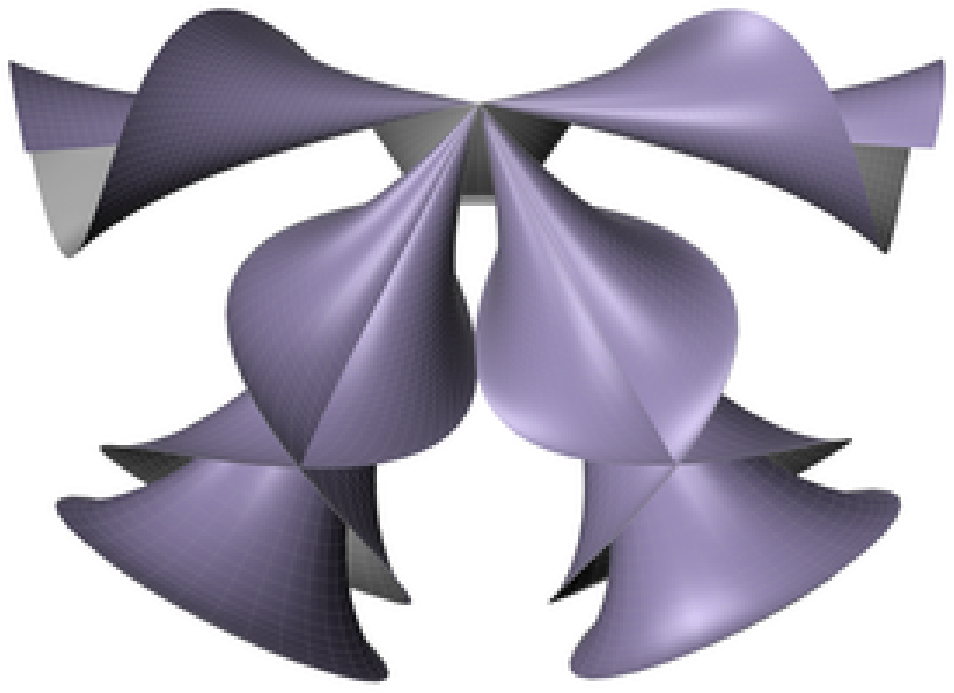}
\caption{Generalized Amsler Surface}
\label{fig:dis1}
\end{figure}

\subsection*{Acknowledgment}
We'd like to thank N. Schmitt for his help in understanding the theory, for the implementation of Toda's algorithm on the computer, and for his beautiful graphics.

\begin{thebibliography}{99}
\bibitem{B} Brander, D., {\it Loop group decompositions in almost split real forms and applications to soliton theory and geometry}, J. Geom. Phys. 58 (2008) 1792-1800.
\bibitem{DH1} Dorfmeister, J. and Haak, G., {\it On symmetries of constant mean curvature surfaces, I: General Theory}, Tohoku Math. J.(2) 50 (1998) 437-454.
\bibitem{DH2} Dorfmeister, J. and Haak, G., {\it On symmetries of constant mean curvature surfaces, Part II: Symmetries in a Weierstrass-type representation}, Int. J. Math., Game Theory Algebra 10 (2000), 121-146.
\bibitem{DIS2} Dorfmeister, J.F. and Sterling I., {\it Symmetric Pseudo-Spherical Surfaces II: Special Examples}, in preparation.
\bibitem{G} Gohberg, I.Ts., {\it A factorization problem in normed rings, functions of isometric and symmetric operators and singular integral equations}, Russian Math. Surv. 19 (1964) 63-114.
\bibitem{IL} Ivey T.A. and Landsberg J.M., {\it Cartan for beginners: differential geometry via moving frames and exterior differential systems}, Graduate Studies in Mathematics Vol. 61, American Mathematical Society, Providence, RI, 2003.
\bibitem{MS} Melko, O. and Sterling, I., {\it Applications of soliton theory to the construction of pseudospherical surfaces in $\mathbb{R}^3$}, Ann. Global Anal. Geom. 11 (1993) 65-107.
\bibitem {S} Spivak, M., {\it A comprehensive introduction to differential geometry}, Volume I, Publish or Perish, 1970.
\bibitem{TU} Terng, C. and Uhlenbeck, K., {\it B\"acklund transformations and loop group actions}, Comm. Pure Appl. Math 53 (2000) 1-75.
\bibitem{Tthesis} Toda, M., {\it Pseudospherical Surfaces by Moving Frames and Loop Groups}, PhD thesis, University of Kansas, 2000.
\bibitem{T} Toda, M., {\it Weierstrass-type Representation of Weakly Regular Pseudospherical Surfaces in Euclidean Space}, Balkan J. of Geom. and Anal. 7 (2002) 87-136.
\bibitem{T2} Toda, M., {\it Initial Value Problems of the Sine-Gordon Equation and Geometric Solutions}, Ann. Global Anal. Geom. 27 (2005) 257-271.
\bibitem{Wu} Wu, H., {\it A Simple Way for Determining the Normalized Potentials 
for Harmonic Maps}, Ann. Global Anal. Geom. 17 (1999) 189-199. 
\bibitem{ZS} Zakharov, V.E. and Shabat,A.B. {\it Integration of the nonlinear equations of mathematical physics by the inverse scattering method II}, Functional Anal. Appl., 13 (3) (1979) 13-22 (Russian), 166-74 (English).
\end{thebibliography}
\end{document}